\documentclass[reqno]{amsart}

\usepackage{enumerate}
\usepackage{amssymb}

\usepackage[colorlinks=true]{hyperref}
\hypersetup{urlcolor=blue, citecolor=cyan}
\hypersetup{citecolor=blue}

\usepackage[applemac]{inputenc}

\numberwithin{equation}{section}

\newtheorem{thm}{Theorem}[section]

\newtheorem{prop}[thm]{Proposition}

\theoremstyle{definition}
\newtheorem{rem}[thm]{Remark}

\newcommand\R{{\mathbb R}}
\newcommand\C{{\mathbb C}}

\newcommand\Z{{\mathbb Z}}
\newcommand\Tor{{\mathbb T}}

\newcommand\Loc{{\mathrm{loc}}}
\newcommand\Vmp{{\zeta}}

\newcommand\goto{\mathop{\longrightarrow}}

\newcommand\MScN[1]{\href{http://www.ams.org/mathscinet-getitem?mr=#1}{\nolinkurl{(#1)}}}
\newcommand\DOI[1]{\href{http://dx.doi.org/#1}{(doi: \nolinkurl{#1})}}
\newcommand\LINK[1]{\href{#1}{(link: \nolinkurl{#1})}}

\newcommand\LAM{\lambda}

\title{Standing waves of the complex Ginzburg-Landau equation}

\author[Thierry Cazenave, Fl\'avio Dickstein and Fred B.~Weissler]{}

\subjclass[2010] {35Q56, 35C08}

 \keywords{Standing waves, complex Ginzburg-Landau equation}

\thanks{Research supported by the ``Brazilian-French Network in Mathematics"}
\thanks{Fl\'avio Dickstein  was partially supported by CNPq (Brasil) and by a ``Research in Paris" grant from the City of Paris.}
\thanks{Fred B. Weissler benefited from a sabbatical leave (CRCT) from the University of Paris 13.}

\begin{document}
\maketitle

\centerline{\scshape Thierry Cazenave}
\medskip
{\footnotesize
 \centerline{Universit\'e Pierre et Marie Curie \& CNRS}
 \centerline{Laboratoire Jacques-Louis Lions}
   \centerline{B.C.~187, 4 place Jussieu}
   \centerline{75252 Paris Cedex 05, France}
   \centerline{email address: {\href{mailto:thierry.cazenave@upmc.fr}{\tt thierry.cazenave@upmc.fr}}}
}

\medskip

\centerline{\scshape Fl\'avio Dickstein}
\medskip
{\footnotesize
 \centerline{Instituto de Matem\'atica}
 \centerline{Universidade Federal do Rio de Janeiro}
   \centerline{ Caixa Postal 68530}
   \centerline{21944--970 Rio de Janeiro, R.J., Brazil}
   \centerline{email address: {\href{mailto:flavio@labma.ufrj.br}{\tt flavio@labma.ufrj.br}}}
}

\medskip

\centerline{\scshape Fred B.~Weissler}
\medskip
{\footnotesize
 \centerline{Universit\'e Paris 13,  Sorbonne Paris  Cit\'e}
 \centerline{CNRS UMR 7539 LAGA}
   \centerline{99 Avenue J.-B. Cl\'e\-ment}
   \centerline{F-93430 Villetaneuse, France}
   \centerline{email address: {\href{mailto:weissler@math.univ-paris13.fr}{\tt weissler@math.univ-paris13.fr}}}
}

\begin{abstract}
We prove the existence of nontrivial standing wave solutions of the complex Ginzburg-Landau equation  $\phi _t = e^{ i\theta } \Delta \phi  + e^{ i\gamma  }  |\phi |^\alpha \phi $ with periodic boundary conditions. Our result includes all values of $\theta $ and $\gamma $ for which $\cos \theta \cos \gamma >0$, but requires that $\alpha >0$ be sufficiently small.
\end{abstract}


\section{Introduction} \label{sIntro} 

We consider the complex Ginzburg-Landau equation
\begin{equation} \label{HGL} 
 \phi _t = e^{ i\theta } \Delta \phi  + e^{ i\gamma  }  |\phi |^\alpha \phi ,
\end{equation} 
where $\alpha >0$,
both on the whole space $\R^N $,  with periodic boundary conditions, and
on a bounded domain $\Omega $ of $\R^N $ with Dirichlet boundary conditions.
We look for standing wave solutions of the form 
\begin{equation} \label{fStaWav} 
\phi (t,x) = e^{i\omega t} u (x)
\end{equation} 
 with $ \omega \in \R$. The resulting equation for $u $ is then
\begin{equation} \label{fEllip1} 
e^{ i\theta }\Delta u + e^{i \gamma } |u |^\alpha u = i\omega  u .
\end{equation} 

Equation~\eqref{HGL} is used to model such phenomena as  superconductivity, chemical turbulence and various types of fluid flows. 
See~\cite{DoeringGHN} and the references cited therein. 
Local and global well-posedness of~\eqref{HGL}, on both $\R^N $ and a domain $\Omega \subset \R^N $, 
are known under various boundary conditions and assumptions on the parameters, see e.g.~\cite{DoeringGL, GinibreVu, GinibreVd, LevermoreO, LevermoreOd, MischaikowM, OkazawaYu, OkazawaYd, OkazawaYt, OkazawaYq}.
Concerning standing wave solutions, the particular case of the nonlinear Schr\"o\-din\-ger equation (i.e. $\theta= \pm \gamma  =\pm \frac {\pi } {2}$) leads to the elliptic equation $\Delta u \pm  |u|^\alpha u \pm \omega u=0$. This equation is the object of a literature too vast to be cited here. 
Another well-known case is $\omega =0$, i.e. stationary solutions. Then necessarily $\gamma =\theta $ modulo $2\pi $ (see Remark~\ref{eComCon}~\eqref{eComCon:2} below) and so equation~\eqref{fEllip1} reduces to $\Delta u+ |u|^\alpha u=0$, which is a special case of the previous equation. 
In the other cases, we are not aware of mathematical results concerning the existence of standing wave solutions. Numerous papers discuss the existence of special solutions (holes, fronts, pulses, sources, sinks, etc), see e.g.~\cite{ChungC, Cruz-PachecoLL, DescalziAT, Doelman, LanGC, LegaF, MancasC, MohamadouNK, PoppSAK, PoppSK, vanSaarloosH}. 

Throughout this paper, all the function spaces are made up of complex-valued functions, but are considered as real Hilbert or Banach spaces. For example, $L^2 (\Omega ) $ is the real Hilbert space of all complex-valued square integrable functions on $\Omega $ with the (real) inner product
\begin{equation} \label{fHyp1} 
( u, v)  _{ L^2 } = \Re \int _\Omega  u  \overline{v}. 
\end{equation} 
In addition, we consider the $N$ dimensional torus $\Tor^N= (\R / 2\pi \Z )^N$ and the space 
\begin{equation} \label{fDenfH} 
H^2 (\Tor^N) = \{ u\in H^2_\Loc (\R^N )  ;\, u  \text{ is $2\pi $-periodic in all variables} \},
\end{equation} 
equipped with the norm of $H^2 (\Omega )$ with $\Omega = (0,2\pi )^N$.

Our first result is the existence of spatially periodic standing wave solutions of~\eqref{HGL} for small $\alpha $.   

\begin{thm} \label{eMain2} 
Suppose  $\gamma ,\theta $ satisfy
\begin{equation} \label{eMain:1} 
- \frac {\pi } {2} < \theta , \gamma < \frac {\pi } {2}. 
\end{equation} 
It follows that there exist  $\alpha _0>0 $ and continuous maps $u : (0,\alpha _0 ) \to H ^2( \Tor^N)$ and $\omega : [0,\alpha _0 ] \to \R$ such that 
for every $\alpha \in (0,\alpha _0)$, $u= u(\alpha )$ is a nontrivial solution of~\eqref{fEllip1} with $\omega =\omega (\alpha )$. 
In particular, the resulting function $\phi $ given by~\eqref{fStaWav} is a standing wave solution of~\eqref{HGL}.

\end{thm} 

Note that it is part of the statement of Theorem~\ref{eMain2} that $ |u|^\alpha u\in L^2_\Loc (\R^N ) $, since both $\Delta u$ and $u$ belong to $L^2_\Loc (\R^N ) $. Our proof of Theorem~\ref{eMain2} proceeds by first constructing solutions of the equation~\eqref{fEllip1} on  the set $\Omega = (0,\pi )^N$ which vanish on the boundary $\partial \Omega $, and then extending these solutions to $\R^N $ by reflection.  Thus, to prove Theorem~\ref{eMain2}, we need first to prove a similar result, but on a bounded domain of $\R^N $, which we now describe. 

We consider a bounded, connected open subset   $\Omega $ of $\R^N $ and we set
\begin{equation} \label{fDefH:1} 
H= \{ u\in H^1_0 (\Omega ) ;\, \Delta u\in L^2 (\Omega )  \},
\end{equation} 
so that $H$ equipped with the scalar product
\begin{equation} \label{fDefH:2} 
(u,v)_H= \Re \int _\Omega u \overline{v}+ \Re \int _\Omega \Delta  u \Delta   \overline{v},  
\end{equation} 
 is a real Hilbert space and $H\hookrightarrow H^1_0 (\Omega ) $. We show the following result.

\begin{thm} \label{eMain} 
Suppose $\Omega $ is a bounded, connected, open subset of $\R^N $.
Let $\gamma ,\theta $ satisfy~\eqref{eMain:1} and let $H$ be defined by~\eqref{fDefH:1}-\eqref{fDefH:2}. 
It follows that there exist  $\alpha _0>0$ and continuous maps $u : (0,\alpha _0 ) \to H $ and $\omega : [0,\alpha _0 ] \to \R$ such that 
for every $\alpha \in (0,\alpha _0)$, $u= u(\alpha )$ is a nontrivial solution of~\eqref{fEllip1} with $\omega =\omega (\alpha )$.
In particular, the resulting function $\phi $ given by~\eqref{fStaWav} is a standing wave solution of~\eqref{HGL}.
\end{thm} 

We prove Theorem~\ref{eMain} by a perturbation argument from the case $\alpha =0$, using  the implicit function theorem.
Indeed, equation~\eqref{fEllip1} with $\alpha =0$ reduces to the eigenvalue problem $- \Delta u  =( e^{i (\gamma -\theta )}- i\omega  e^{- i\theta })u$. 
As is well known, all eigenvalues of $-\Delta $ with Dirichlet boundary conditions are positive real numbers, while $\lambda ^\star= : e^{i (\gamma -\theta )}- i\omega  e^{- i\theta }$ is not in general real. 
It turns out that $\lambda ^\star$ is real precisely when $\omega = \frac {\sin (\gamma -\theta )} {\cos \theta }$, in which case $\lambda ^\star= \frac {\cos \gamma } {\cos \theta }$. 
Unfortunately, this value of $\lambda ^\star $ is not always an eigenvalue of $-\Delta $. 
To overcome this problem, we introduce another parameter $\mu >0$ and consider the equation
\begin{equation} \label{fEqV1} 
 \Delta v +\mu e^{i(\gamma -\theta )}  |v|^\alpha  v -i\omega e^{-i \theta }v =0.
\end{equation} 
(See also Remark~\ref{eMu0eq1}.)
Note that if $\alpha >0$, then a solution of~\eqref{fEqV1} can be turned into a solution of~\eqref{fEllip1}  by a simple multiplicative factor. 

Equation~\eqref{fEqV1} in the case $\alpha =0$ now becomes $-\Delta v= (\mu e^{i(\gamma -\theta )} -i\omega e^{-i \theta } ) v$. Given any $\lambda >0$, one sees that  $\mu _0 e^{i(\gamma -\theta )} -i\omega _0 e^{-i \theta } =\lambda $ if and only if  
\begin{equation}  \label{fHyp3} 
\omega _0  =   \LAM\frac {\sin (\gamma -\theta )} {\cos \theta },
\end{equation} 
and
\begin{equation}  \label{fHyp4} 
\mu _0 = \LAM \frac {\cos \theta } {\cos \gamma }.
\end{equation} 
The implicit function theorem now yields the following result.

\begin{thm} \label{eTech1} 
Suppose $\Omega $ is a bounded, connected, open subset of $\R^N $ and let $H$ be defined by~\eqref{fDefH:1}-\eqref{fDefH:2}.
Let
\begin{equation}  \label{fHyp1:1} 
\LAM = \LAM( -\Delta ),
\end{equation} 
be an  eigenvalue of $-\Delta $ in $L^2 (\Omega ) $ with domain $H$ and $\varphi $ a corresponding eigenvector such that
\begin{equation} \label{fHyp1:1b1}
\int _\Omega | \varphi| ^2=1.
\end{equation} 
Assume that $\lambda $ is a simple eigenvalue, in the sense that the corresponding eigenspace is $\C \varphi $. (For instance, $\LAM$ can be the first eigenvalue of $-\Delta $.)
Let $\gamma ,\theta $ satisfy~\eqref{eMain:1} and let $\omega _0$ and $\mu _0$ be defined by~\eqref{fHyp3}-\eqref{fHyp4}.  
It follows that there exist $\alpha _0>0$ and continuous maps $v : [0,\alpha _0 ] \to H$, $\mu : [0,\alpha _0 ] \to \R$ and $\omega : [0,\alpha _0 ] \to \R$ such that $v(0) =  \varphi  $, $\mu (0)= \mu_0$, $\omega (0)=  \omega _0 $ and such that~\eqref{fEqV1} holds for  all $0\le \alpha \le \alpha _0 $.
\end{thm} 

The above results call for several remarks. 
Since our proof of Theorem~\ref{eTech1} is based on a perturbation argument, we have no information on the size of $\alpha _0$. In addition, based on what is known about standing waves of the nonlinear Schr\"o\-din\-ger equation and stationary solutions of the nonlinear heat equation, one would expect that, at least in space dimension $N\ge 2$, there would exist an infinite family of standing wave solutions (all with the same $\omega $). Our results do not address this question at all. Another important issue is the stability (both linear and dynamical) of the standing waves. 

We next make a few remarks concerning the conditions on $\theta $ and $\gamma $.

\begin{rem} \label{eComCon} 
Suppose for example equation~\eqref{fEllip1} is set on a bounded, connected subset $\Omega $ of $\R^N $ with  Dirichlet boundary conditions. (Similar calculations can be made in the case of periodic boundary conditions.)
Let $u\in H^1_0 (\Omega ) \cap L^{\alpha +2}(\Omega )$ with $\Delta u\in L^2 (\Omega ) $, $u\not \equiv 0$ be a solution of~\eqref{fEllip1}.  
Since
\begin{equation} \label{fIntPart} 
\int _\Omega  \overline{u}\Delta u= -\int _\Omega  |\nabla u|^2,
\end{equation} 
we deduce from~\eqref{fEllip1} that 
\begin{equation} \label{fNEZ1} 
e^{i\theta }\int _\Omega  |\nabla u|^2= e^{i \gamma } \int _\Omega  |u|^{\alpha +2}
- i \omega \int _\Omega  |u|^2.
\end{equation} 
We now can draw the following consequences.
\begin{enumerate}[{\rm (i)}] 

\item \label{eComCon:1} 
Considering the real part of~\eqref{fNEZ1}, we obtain
\begin{equation} \label{fNE1} 
\cos \theta \int _\Omega  |\nabla u|^2= \cos \gamma  \int _\Omega  |u|^{\alpha +2}. 
\end{equation} 
Thus we see that either $\cos \theta =\cos \gamma =0$ or else $\cos \theta \cos \gamma >0$.
If $\cos \theta =\cos \gamma =0$, then the equation~\eqref{HGL} becomes the nonlinear Schr\"o\-din\-ger equation $ i \phi _t = \pm   \Delta \phi  \pm   |\phi |^\alpha \phi $, whose standing wave solutions have been extensively studied.  Assume now  $\cos \theta \cos \gamma >0$. Changing $(\gamma ,\theta , \omega )$ to $(\gamma +\pi , \theta +\pi , -\omega )$ leaves the equation invariant, so we may assume that $\cos \gamma >0$ and $\cos \theta >0$. Therefore, we may assume without loss of generality that~\eqref{eMain:1} holds.

\item \label{eComCon:2} 
It follows easily from~\eqref{fNEZ1} and~\eqref{eMain:1}  that $\omega =0$ (i.e. $u$ is a stationary solution of~\eqref{HGL}) if and only if 
$\gamma =\theta $ modulo $2\pi $.

\item \label{eComCon:3} 
Finally, observe that changing $(u,\gamma , \theta , \omega )$ to $( \overline{u} , -\gamma , -\theta , -\omega )$ leaves the equation invariant.
\end{enumerate} 
\end{rem} 

The next remark gives some variants of the main results.

\begin{rem} \label{eRem10} 
\begin{enumerate}[{\rm (i)}] 

\item \label{eRem10:1} 
If $u$ is a periodic solution  of~\eqref{fEllip1}, as in Theorem~\ref{eMain2}, then for every positive integer $n$, $u_n (x)= n^{\frac {2} {\alpha }} u(nx)$  is also a solution of~\eqref{fEllip1} in $H ^2( \Tor^N)$, but with $\omega  $ replaced by $n^2\omega $.  In this way, we obtain infinitely many solutions on $\Tor^N$ (starting from one given solution), but for values of $\omega  $ that change with the solution.

\item \label{eRem10:3} 
In Theorem~\ref{eMain2}, we may replace the torus $(\R / 2\pi \Z)^N$ by $(\R/ 2\ell_1\Z) \times \cdots \times ( \R/2 \ell_N \Z)$, where $\ell_1, \dots, \ell_N>0$. It suffices to apply Theorem~\ref{eMain} with $\Omega = (0, \ell_1) \times \cdots \times (0, \ell_N )$ instead of $\Omega =(0,\pi )^N$. The above remark about rescaling applies in this situation as well.

\item \label{eRem10:4} 
In the case where $\Omega $ is the unit ball of $\R^N $, there is version of Theorem~\ref{eTech1} in the space $L^2 _{ {\mathrm{rad}} } (\Omega )$ of radially symmetric functions.
Note that all the eigenvalues of $-\Delta $ in $L^2 _{ {\mathrm{rad}} } (\Omega )$ with Dirichlet boundary conditions are simple. 
Thus for every integer $n$, there exists $\alpha _0>0$ such that for $0<\alpha <\alpha _0$ there exist $n$ different standing wave solutions of~\eqref{HGL}  (assuming $\gamma \not = \theta $). Indeed, the solutions are different because for sufficiently small $\alpha $, the corresponding value of $\omega $ is close to $\omega _0$ given by~\eqref{fHyp3}; and the values of $\omega _0$ corresponding to different eigenvalues are all different. It would be interesting to know if these solutions are related by dilation, as is true for the eigenfunctions of $-\Delta $ in $L^2 _{ {\mathrm{rad}} } (\Omega )$. 

\end{enumerate} 
\end{rem} 

\begin{rem} 
The method we use to prove
Theorem~\ref{eMain} is not valid in the case $\Omega =\R^N $. 
In fact, the conclusion of Theorem~\ref{eMain}  is false if $\Omega =\R^N $. Indeed, suppose $\theta =\gamma $, in which case $\omega =0$ by the same argument as in Remark~\ref{eComCon}~\eqref{eComCon:2}. 
In this case, equation~\eqref{fEllip1} becomes $-\Delta u=  |u|^\alpha u$, which has no nontrivial solutions in $H^1 (\R^N ) $ if $\alpha $ is Sobolev subcritical (i.e. $(N-2)\alpha <4$).
This fact is a consequence of the Poho\v zaev identity~\cite{Pohozaev}. 
For the precise formulation needed here, see~\cite[Corollary~1, p.321]{BerestyckiL}.
\end{rem}

\section{Proof of Theorem~$\ref{eTech1}$}
In this section, we prove Theorem~\ref{eTech1}. 

\begin{proof} [Proof of Theorem~$\ref{eTech1}$]
It follows from~\eqref{fHyp3}-\eqref{fHyp4}  that
\begin{equation} \label{fHyp5} 
\mu  _0e^{i(\gamma -\theta )}-i\omega _0e^{-i\theta }=\LAM.
\end{equation} 
Therefore, $\varphi $ is a solution of~\eqref{fEqV1} with $\alpha =0$, $\mu =\mu _0$ and $\omega =\omega _0$.   For $\alpha >0$ small, $\mu $ and $\omega $ close to $\mu _0$ and $\omega _0$, we seek a solution $v$ of~\eqref{fEqV1} of the form $v= \varphi + \Vmp$ with
$\Vmp \in H_1$, where $H_1$ is the orthogonal complement of $\C \varphi $ in $H$. 
The main tool we use is the implicit function theorem. The first order of business is to define an appropriate mapping $F$.
 We fix
\begin{equation}  \label{fHyp16} 
0<\widetilde{\alpha } <\infty   \text{ such that }(N-2)\widetilde{\alpha }  \le 2, 
\end{equation} 
so that $H \hookrightarrow L^{2( \widetilde{\alpha } +1)} (\Omega ) $ by Sobolev's embedding.
We set
\begin{equation} \label{fHyp17} 
X= \R ^2 \times H_1,
\end{equation} 
and we define the  map $F: (-\infty , \widetilde{\alpha } ] \times X \to L^2 (\Omega ) $ by
\begin{align} 
F(\alpha , \mu , \omega , \Vmp ) & = \Delta v  + \mu e^{i(\gamma -\theta )}  g(\alpha ,v) -i \omega e^{ -i \theta }v, \label{fHyp14} \\
v &= \varphi +\Vmp, \label{fHyp14v2} 
\end{align} 
where the function $g:\R\times \C \to \C$ is given by
\begin{equation} \label{fDefg}
g(\alpha ,v) = \begin{cases} 
 |v|^\alpha v & \alpha >0, \\
 v & \alpha \le 0. 
\end{cases}   
\end{equation} 
It follows  that if $F(\alpha , \mu , \omega , \Vmp)=0$ and $\alpha >0$, then $v  $ is a solution of equation~\eqref{fEqV1}.
Since
\begin{equation}  \label{fHyp1:3} 
\Delta \varphi  + \LAM \varphi =0,
\end{equation}  
we deduce from~\eqref{fHyp5} that
\begin{equation}  \label{fHyp14:1} 
F(0 , \mu _0, \omega _0, 0 ) = 0.
\end{equation} 
It follows (for instance from Proposition~\ref{eDiffer1} below)
that  the map $(\alpha ,v) \mapsto  g(\alpha ,v(\cdot ))  $ is continuous $(-\infty ,  \widetilde{\alpha }]\times  H \to L^2 (\Omega ) $, from which we deduce that $F$ is continuous $(-\infty , \widetilde{\alpha } ] \times X \to L^2 (\Omega )$.
Furthermore, if $ \alpha \le   \widetilde{\alpha }$, then by Proposition~\ref{eDiffer1} the map $v\mapsto  g(\alpha ,v(\cdot ))  $ is differentiable everywhere on $H$, so that  the map 
$( \mu , \omega , \Vmp ) \mapsto F(\alpha , \mu , \omega , \Vmp ) $ is differentiable everywhere.
In addition, using~\eqref{fDParv1}, we have
\begin{align}
\frac {\partial F} {\partial \mu }(\alpha , \mu , \omega , \Vmp )&=e^{i(\gamma -\theta )} g(\alpha ,v (\cdot ))  ,  \label{fHyp20} \\
\frac {\partial F} {\partial \omega  } (\alpha , \mu , \omega , \Vmp) &=  
-i e^{ -i \theta } v,\label{fHyp21} 
\end{align}  
and 
\begin{multline} \label{fHyp22} 
\frac {\partial F} {\partial \Vmp }(\alpha , \mu , \omega , \Vmp) w \\= 
\begin{cases} 
\Delta w   +\mu e^{i(\gamma -\theta )}[ |v|^\alpha w + \alpha  |v|^{\alpha -2} v \Re ( \overline{v}w)]  -i \omega e^{-i\theta }w  & \alpha >0,\\ 
\Delta w +  [\mu e^{i(\gamma -\theta )}  -i \omega e^{-i\theta }] w  & \alpha \le 0.
\end{cases} 
\end{multline} 
We now show that the derivative
\begin{equation*} 
  \frac {\partial F } {\partial (\mu ,\omega ,\Vmp)}  (0 , \mu _0, \omega _0, 0): X \to L^2 (\Omega ) 
\end{equation*} 
 is a bijection. 
Indeed, we deduce from~\eqref{fHyp20}--\eqref{fHyp22}  that 
\begin{align}
\frac {\partial F} {\partial \mu }(0 , \mu _0, \omega _0, 0 )&=e^{i(\gamma -\theta )}\varphi   ,  \label{fHyp20:1} \\
\frac {\partial F} {\partial \omega  } (0 , \mu _0, \omega _0, 0 ) &=  
-i e^{ -i \theta } \varphi  ,\label{fHyp21:1} \\
\frac {\partial F} {\partial \Vmp } (0 , \mu _0, \omega _0, 0 ) w&= 
\Delta w  +  \LAM w .\label{fHyp22:1} 
\end{align}  
where we used~\eqref{fHyp5} in the last identity.  
Therefore,
\begin{equation} \label{fHyp26} 
 \Bigl[ \frac {\partial F} {\partial (\mu  ,\omega  ,\Vmp  ) }(0 , \mu _0, \omega _0, 0 ) \Bigr] \cdot  (a,b,w)=  A (a,b,w),
\end{equation}  
where
\begin{equation} \label{fHyp27} 
A (a,b,w)= 
  ae^{i(\gamma -\theta )}\varphi   -i b e^{ -i \theta } \varphi  +\Delta w+ \LAM w.
\end{equation} 
We first claim that the kernel of $A$  is trivial. Indeed, suppose
\begin{equation} \label{fHyp28} 
A (a,b,w)= 0.
\end{equation} 
Multiplying the equation~\eqref{fHyp28}  by $ \overline{\varphi}  $, integrating by parts on $\Omega $ and using~\eqref{fHyp1:1b1}  and~\eqref{fHyp1:3}, we obtain
$0=    ae^{i(\gamma -\theta )}  -i b e^{ -i \theta }$, i.e. $ae^{i\gamma }= ib$. 
Using~\eqref{eMain:1}, we conclude that $a=b=0$.
It then follows from~\eqref{fHyp28} that $\Delta w+\LAM w=0$, so that $w\in \C \varphi $. Since $\C \varphi \cap H_1 =\{0\}$, this proves the claim.  

We next claim that $A$ is surjective. Let $M(z)=ae^{i(\gamma -\theta )}  -i b e^{ -i \theta }$, for $z=a+bi$. Considering $\C$ as a real linear space, we see that $M$ is a linear operator $\C \to \C$. As shown above, $\ker{M}=\{0\}$, and so  M is a bijection.
Thus, given $f\in L^2 (\Omega ) $ there exist  $a, b\in \R$ such that
\begin{equation}  \label{fHyp29} 
 a e^{i(\gamma -\theta )}  -i b e^{-i \theta } =\int _\Omega f \overline{\varphi}  .
\end{equation}   
It follows from~\eqref{fHyp1:1b1} and~\eqref{fHyp29} that $f- a e^{i(\gamma -\theta )} \varphi  + i  b e^{-i \theta } \varphi $ belongs to the orthogonal of $\C\varphi $. Therefore, there exists a unique  $w \in H_1 $ such that
\begin{equation*} 
\Delta w+ \LAM w = f - a e^{i(\gamma -\theta )} \varphi + i  b e^{-i \theta } \varphi,
\end{equation*} 
i.e. $A(a,b, w)=f$. This proves surjectivity.

At this point, we wish to apply the implicit function theorem~\cite[Theorem~ 4.B, p.150]{Zeidler}. 
The only condition that we have not yet verified is that the map $\partial _{ (\mu ,\omega ,\Vmp)} F$ given by 
\begin{equation} \label{fHyp30} 
\begin{cases} 
(-\infty , \widetilde{\alpha } ] \times X \to {\mathcal L}(X, L^2 (\R^N ) ) \\ \displaystyle 
(\alpha ,\mu ,\omega , \Vmp) \mapsto   \frac {\partial F } {\partial (\mu ,\omega ,\Vmp)}  (\alpha ,\mu ,\omega , \Vmp)
\end{cases} 
\end{equation} 
is continuous at the point $(0, \mu _0, \omega _0, 0)$. 
This is an immediate consequence of Proposition~\ref{eDiffer2}, since $ \varphi \ne 0$ a.e. in $\Omega $. To see this last property, we note that $\varphi $ is analytic in the connected, open set $\Omega $, see e.g.~\cite{Friedman}, so that it cannot vanish on a set of positive measure.  

By the above cited the implicit function theorem,  there exist $\alpha _0 >0$ and continuous  maps $\Vmp : [0,\alpha _0 ] \to H_1 $, $\mu : [0,\alpha _0 ] \to \R$ and $\omega : [0,\alpha _0 ] \to \R$ such that $\Vmp (0) = 0  $, $\mu (0)= \mu_0$, $\omega (0)=  \omega _0 $ and such that
$F(\alpha , \mu (\alpha ), \omega (\alpha ), \Vmp (\alpha ))  =0$
for $0\le \alpha \le \alpha _0 $. This completes the proof (with $v(\alpha )= \varphi + \Vmp (\alpha )$).
\end{proof} 

\begin{rem} \label{eMu0eq1} 
The parameter $\mu $ is not only useful to ensure that the equation~\eqref{fEqV1} has the form $\Delta v+ \lambda v=0$ when $\alpha =0$, where $\lambda $ is an eigenvalue of $-\Delta $. 
It also provides a second parameter in the implicit function theorem so that the linearized operator is  bijective.
\end{rem} 

\section{Proof of Theorems~\ref{eMain} and~\ref{eMain2}}

\begin{proof} [Proof of Theorem~$\ref{eMain}$]
We apply Theorem~\ref{eTech1}. 
(Note that there always exists $\LAM$ as in the statement, for example
$\LAM$ can be the first eigenvalue of $-\Delta $.)
Observe that 
\begin{equation} \label{fHyp2:1} 
\cos \gamma  > 0,\quad  \cos \theta >0,
\end{equation} 
by~\eqref{eMain:1}, so that $\mu_0> 0$ by \eqref{fHyp4}. Therefore, we may choose $\alpha _0$ small enough so that $\mu(\alpha)> 0$ for $\alpha\in  [0,\alpha _0 ] $.  
Thus $u= \mu^{\frac {1} {\alpha }}v$ is well defined and satisfies~\eqref{fEllip1}
(since $v$ satisfies~\eqref{fEqV1}).
Since $v(0)= \varphi $, we have $v(\alpha )\not \equiv 0$, hence $u(\alpha )\not \equiv 0$, for $\alpha >0$ sufficiently small. Thus, by choosing $\alpha _0>0$ possibly smaller, we have $u(\alpha )\not \equiv 0$ for $0<\alpha <\alpha _0$. 
\end{proof} 

\begin{proof} [Proof of Theorem~$\ref{eMain2}$]
We apply Theorem~\ref{eMain} with $\Omega = (0,\pi )^N$, and we obtain $0< \alpha _0 < \frac {2} {(N-2)_+}$ and continuous maps $ \widetilde{u}  : (0,\alpha _0 ) \to H $ (defined by~\eqref{fDenfH}) and $\omega : [0,\alpha _0 ] \to \R$ such that  for every $\alpha \in (0,\alpha _0)$, $ \widetilde{u} =  \widetilde{u} (\alpha )$ is a solution of~\eqref{fEllip1} on $\Omega $. We now extend $ \widetilde{u} $ to $\R^N $ by symmetry. More precisely, given $x\in \R^N $, there exists a unique family of integers $(k_j) _{ 1\le j\le N }$ such that $k_j\pi\le x_j < (k_j +1)\pi$, and we set 
\begin{equation} 
u(x)= (-1)^{\sum_{ j=1 }^N k_j } \widetilde{u}(  \widetilde{x} ), 
\end{equation} 
where $ \widetilde{x}_j= x_j- k_j\pi $. 
It follows that $u\in H^1 _\Loc(\R^N ) \cap L^{2(\alpha +1)} _\Loc (\R^N ) $ is a solution of~\eqref{fEllip1} on $\R^N $, and by standard elliptic regularity, $u\in H^2 _\Loc(\R^N )$. Since $u$ is clearly $2\pi $-periodic in all variables, we see that $u\in H ^2( \Tor^N)$. 
\end{proof} 

\begin{rem} 
\begin{enumerate}[{\rm (i)}] 

\item Recall that $u(\alpha )$ constructed in the proof of Theorem~\ref{eMain} is given by 
$u (\alpha )= \mu (\alpha )^{\frac {1} {\alpha }}v (\alpha )$, where the functions $v(\alpha )$ and $\mu (\alpha )$ are given by Theorem~\ref{eTech1}. Furthermore, $\mu (\alpha )\to \mu _0$ given by~\eqref{fHyp4} as $\alpha \to 0$.  Clearly, if $\mu _0>1$ (i.e. $\lambda > \frac {\cos \gamma } {\cos \theta }$), then $ \| u(\alpha )\| _{ L^2 }\to \infty $ as $\alpha \to 0$, while if $\mu _0<1$ (i.e. $\lambda < \frac {\cos \gamma } {\cos \theta }$), then $ \| u(\alpha )\| _{ L^2 }\to 0$ as $\alpha \to 0$.
In the latter case, the curve $u(\alpha )$ bifurcates from the trivial branch of solutions of~\eqref{fEllip1} at $\alpha =0$.  The same conclusions hold for the solutions constructed in Theorem~\ref{eMain2}.  

\item In the context of Theorem~\ref{eTech1}, suppose $\lambda $ and $ \widetilde{\lambda } $ are two different simple eigenvalues of $-\Delta $ with corresponding eigenvectors $\varphi $ and $ \widetilde{\varphi } $. Let $v,\mu , \omega $ and $ \widetilde{v} ,  \widetilde{\mu } ,  \widetilde{\omega } $ be the resulting continuous maps constructed by Theorem~\ref{eTech1}.  Since $v(\alpha ) \to \varphi $ and $ \widetilde{v} (\alpha ) \to  \widetilde{\varphi}  $ as $\alpha \to 0$, it is clear that $v(\alpha )\not =  \widetilde{v} (\alpha ) $ for $\alpha >0$ small. 
Similarly, it is clear (see~\eqref{fHyp3}) that  $\omega (\alpha )\not =  \widetilde{\omega } (\alpha ) $ for $\alpha >0$ small. 
The functions $u (\alpha )= \mu (\alpha )^{\frac {1} {\alpha }}v (\alpha )$ and $ \widetilde{u}  (\alpha )=  \widetilde{\mu}  (\alpha )^{\frac {1} {\alpha }} \widetilde{v}  (\alpha )$ are solutions, respectively, of equation~\eqref{fEllip1} and of equation~\eqref{fEllip1} with $ \widetilde{\omega } $ instead of $\omega $. Thus it is also clear that $u (\alpha )\not =  \widetilde{u } (\alpha ) $ for $\alpha >0$ small. 

\item 
Suppose $\Omega $ is the unit ball of $\R^N $ and consider radially symmetric standing waves. As noted above (see Remark~\ref{eRem10}~\eqref{eRem10:4}) all the eigenvalues $0<\lambda _1<\lambda _2< \cdots$ of  $-\Delta $ in $L^2 _{ {\mathrm{rad}} } (\Omega )$ with Dirichlet boundary conditions are simple. And so, Theorem~\ref{eTech1} can be applied with each $\lambda =\lambda _k$ for each $k\ge 1$, thus producing an infinite family of curves $u_k(\alpha )$ of standing waves. 
If $\lambda _k < \frac {\cos \gamma } {\cos \theta }$ (which can happen only for finitely many $k$), then $ \| u_k (\alpha )\| _{ L^2 }\to 0$ as $\alpha \to 0$. On the other hand, if $\lambda _k > \frac {\cos \gamma } {\cos \theta }$ (which necessarily happens  infinitely many $k$), then $ \| u_k (\alpha )\| _{ L^2 }\to \infty $ as $\alpha \to 0$.
Observe that the number of eigenvalues such that $\lambda _k < \frac {\cos \gamma } {\cos \theta }$ obviously depends on $\theta $ and $\gamma $.

\end{enumerate} 
\end{rem}

\appendix

\section{An extension of the map $(\alpha ,v) \mapsto  |v|^\alpha v$} \label{aDiff} 
In this section we  construct an explicit extension of the map $(\alpha ,v) \mapsto  |v|^\alpha v$ to include negative values of $\alpha $ and we study its differentiability with respect to $v$. 

We consider the function $g: \R\times \C\to \C$ defined by~\eqref{fDefg}
and we  define $H: \R \times \C \times \C \to \C$ by
\begin{equation} \label{fDefH}
H(\alpha ,v,u) = \begin{cases} 
|v|^\alpha u+ \alpha |v|^{\alpha -2} v \Re (  \overline{v} u ) & \alpha >0, v\not = 0,\\ 
0 & \alpha >0, v=0, \\
u & \alpha \le 0 .
\end{cases}   
\end{equation} 
It follows easily that $g\in C(\R\times \C , \C)$  and 
$H$ is continuous, except at the points $(0,0, u)$ with $u\not = 0$ (where it is discontinuous).
Moreover,  $g$ is differentialble with respect to $v$ at every point $(\alpha ,v)\in \R\times \C$
 (where $\C$ is considered as a real Hilbert space),  and
\begin{equation}  \label{fGrpv} 
\partial _v  g(\alpha ,v)u = H(\alpha ,v,u),
\end{equation} 
for all $\alpha \in  \R$ and $u,v\in \C$.

Let $\Omega $ be a bounded open subset of $\R^N $ and, given $1\le r\le \infty $, let $L^r (\Omega ) $   be the usual Lebesgue space of complex valued functions, equipped with its standard norm $ \| \cdot  \| _{ L^r }$, considered as a real Banach space. 
We fix $a>0$ and set $p=2(a+1)$. Given  $v\in L^p (\Omega ) $, we define
\begin{equation} \label{fDG} 
G(\alpha ,v) (\cdot ) =  g(\alpha , v(\cdot )),
\end{equation} 
where $g$ is given by~\eqref{fDefg}.

\begin{prop} \label{eDiffer1} 
The map $G$ is continuous $(-\infty ,a] \times L^{p} (\Omega ) \to   L^2 (\Omega ) $. 
Moreover, $G$ is Fr\'echet differentiable with respect to $v$ everywhere on $(-\infty ,a] \times L^{p} (\Omega )$ and 
\begin{equation} \label{fDParv}
\partial _v G(\alpha ,v)= L _{ \alpha ,v },  
\end{equation} 
with
\begin{equation} \label{fDParv1}
L _{ \alpha ,v } u= H(\alpha ,v(\cdot ), u(\cdot )),  
\end{equation} 
for all $\alpha \in (-\infty ,a) $, $v\in L^p (\Omega ) $ and $u\in L^2 (\Omega ) $, where $H$ is defined by~\eqref{fDefH}.
\end{prop}

\begin{proof}
Since $\Omega $ is bounded, it follows from the estimate
\begin{equation} \label{fEsta2} 
 |g(\alpha  ,v)| \le 
 \begin{cases} 
 |v|^{\alpha  +1} & \alpha  \ge 0,\\
 |v| & \alpha  <0, 
 \end{cases} 
\end{equation} 
that $G(\alpha ,v)\in L^2 (\Omega ) $ for all $\alpha <a$. Moreover, it follows easily 
from~\eqref{fEsta2} and the dominated convergence theorem that $G\in C((-\infty ,a)\times L^p (\Omega ), L^2 (\Omega ) )$. 
Next, we deduce from the estimate
\begin{equation} \label{fEstv} 
| H(\alpha ,v,u) | \le 
 \begin{cases} 
 (\alpha +1) |v|^{\alpha }  |u| & \alpha > 0, \\
 |u| & \alpha \le 0,
 \end{cases} 
\end{equation} 
 that 
\begin{equation*} 
 \|  L _{ \alpha ,v } u\| _{ L^2 } \le 
\begin{cases} 
 (\alpha +1)  \| v \| _{ L^p }^\alpha  |\Omega |^{\frac {a-\alpha } {a+1}}  \| u\| _{ L^p } &  0 < \alpha <a, \\
   |\Omega |^{\frac {a} {a+1}}  \| u \| _{ L^p   }& \alpha \le 0.
\end{cases} 
\end{equation*} 
Therefore, given any $\alpha <a$ and $v\in L^p (\Omega )$, we see that $L _{ \alpha ,v }$ is linear\footnote{Recall that we consider the space $L^2 (\Omega ) $ of complex-valued functions as a  {\bf real} Banach space.} and continuous $L^p (\Omega ) \to L^2 (\Omega ) $. 
It is not difficult to verify, using~\eqref{fEstv} and the dominated convergence theorem  that
\begin{equation} \label{fCont1} 
(\alpha ,v) \mapsto L _{ \alpha ,v } \text{ is continuous } ((-\infty ,a)\setminus \{0\}) \times L^p (\Omega )\to  {\mathcal L}( L^p (\Omega ), L^2 (\Omega ))  .
\end{equation}  
Given $0<\alpha <a$, $v, u \in L^p (\Omega )$ and $0<t\le 1$, it follows from~\eqref{fGrpv}   that
\begin{equation} \label{fFin1} 
\begin{split} 
 G(\alpha  , v + u) - G(  \alpha ,v )   -  L _{ \alpha ,v } u & = 
\int _0^1 H(\alpha ,v +\sigma u,u) \, d\sigma -  L _{ \alpha ,v } u\\ &  =
 \int _0^1 [ L_{\alpha ,v +\sigma u} u -  L _{ \alpha ,v } u]\, d\sigma. 
\end{split} 
\end{equation} 
Applying~\eqref{fCont1}, we deduce  that
\begin{equation} \label{fFin2} 
 \Bigl\|   \int _0^1 [ L_{\alpha ,v +\sigma u} u -  L _{ \alpha ,v } u]\, d\sigma \Bigr\| _{ L^2 } = o ( \|u\| _{ L^p })\quad  \text{as}\quad  \|u\| _{ L^p }\to 0. 
\end{equation} 
Thus we see that $G$ is Fr\'echet differentiable with respect to $v $ at $(\alpha , v)$ with derivative $L _{ \alpha ,v }$, provided $0<\alpha <a$. Since $G(\alpha , v)=v$ if $\alpha \le 0$, we see that $G (\alpha ,v)$ is differentiable at $(\alpha ,v)$ for any $L^p (\Omega ) $, with derivative $\partial _v G(0,v)= I$.  This completes the proof.
\end{proof} 

\begin{prop} \label{eDiffer2} 
Let $v \in L^p (\Omega )$ satisfy $v(x)\not = 0$ for a.a. $x\in \Omega $. 
It follows that  $\partial _v G (\alpha ,v)$ is continuous at $(0,v)$. 
\end{prop} 
\begin{proof} 
Consider a sequence $(\alpha _n,v_n)_{n\ge 1}\subset \R\times L^p (\Omega )$ such that $\alpha _n\to 0$ in $\R$ and $v_n\to v$ in $L^p (\Omega )$. 
By~\eqref{fDParv}, we need to show that  ${ L}_{\alpha _n,v_n}\to { L}_{0,v}=I$ in ${\mathcal L}( L^p (\Omega ), L^2 (\Omega ))$ as $n\to \infty$. 
Since $L _{ \alpha ,v }=I$ if $\alpha \le 0$, we may assume that $\alpha _n>0$. We write
\begin{equation} \label{fDiffer21} 
{ L}_{\alpha _n,v_n }={ L}^1_{\alpha _n,v_n }+{ L}^2_{\alpha _n,v_n },
\end{equation} 
where 
\begin{equation*}
{ L}^1_{\alpha _n,v _n }u=   |v_n |^{\alpha _n}u , 
 \quad { L}^2_{\alpha _n,v _n}u=
\alpha_n |v_n |^{\alpha _n -2} v_n  \Re (  \overline{v_n } u ) . 
\end{equation*} 
 We first note that  $\alpha _n |v_n|^{\alpha _n}\le \alpha _n (1+ |v_n |^a)$,
so that
\begin{equation} \label{fDiffer24} 
 \|  { L}^2_{\alpha _n,v_n } \|  _{  {\mathcal L }(L^p, L^2)  } \le \alpha _n ( |\Omega |^{\frac {a} {p}} +  \| v _n \| _{ L^p }^a) \goto _{ n\to \infty  }0. 
\end{equation} 
Furthermore, since $|v|>0$ a.e. in $\Omega $, we see that 
\begin{equation*} 
|v_n|^ {\alpha_n }  \goto  _{ n\to \infty  }1,
\end{equation*} 
a.e. in $\Omega $, and  
it follows by dominated convergence that  
\begin{equation} \label{fDiffer25}
{ L}^1_{\alpha _n,v_n}\goto  _{ n\to \infty  } I,
\end{equation} 
in ${\mathcal L} (L^p (\Omega ), L^2 (\Omega ) )$. The result is now a consequence of \eqref{fDiffer21}--\eqref{fDiffer25}.
\end{proof}

\end{document}